\RequirePackage{fix-cm}
\documentclass[smallextended,envcountsect]{svjour3}       
\smartqed  
\usepackage{graphicx}
\usepackage{amsmath,amsxtra,amssymb,latexsym,amscd}
\usepackage{graphicx,color}
\usepackage[active]{srcltx}
\usepackage[utf8]{inputenc}
\usepackage[mathscr]{euscript}
\usepackage{mathrsfs,cite}
\usepackage[english]{babel}
\usepackage{color,marvosym}

\def\ar{a\kern-.370em\raise.16ex\hbox{\char95\kern-0.53ex\char'47}\kern.05em}
\def\ees{{\accent"5E e}\kern-.385em\raise.2ex\hbox{\char'23}\kern-.08em}
\def\eex{{\accent"5E e}\kern-.470em\raise.3ex\hbox{\char'176}}
\def\AR{A\kern-.46em\raise.80ex\hbox{\char95\kern-0.53ex\char'47}\kern.13em}
\def\EES{{\accent"5E E}\kern-.5em\raise.8ex\hbox{\char'23 }}
\def\EEX{{\accent"5E E}\kern-.60em\raise.9ex\hbox{\char'176}\kern.1em}
\def\ow{o\kern-.42em\raise.82ex\hbox{
		\vrule width .12em height .0ex depth .075ex \kern-0.16em \char'56}\kern-.07em}
\def\OW{O\kern-.460em\raise1.36ex\hbox{
		\vrule width .13em height .0ex depth .075ex \kern-0.16em \char'56}\kern-.07em}
\def\UW{U\kern-.42em\raise1.36ex\hbox{
		\vrule width .13em height .0ex depth .075ex \kern-0.16em \char'56}\kern-.07em}
\def\DD{D\kern-.7em\raise0.4ex\hbox{\char '55}\kern.33em}

\begin{document}

\title{On the existence  of Pareto solutions for polynomial vector optimization problems 
\thanks{The research of the first and the third author  was supported by the National Research Foundation of Korea Grant funded by the Korean Government (NRF-2016R1A2B4011589). The research of the second  author is supported by the National Foundation for Science and Technology Development (NAFOSTED), Vietnam grant 101.04-2016.05.}
}

\titlerunning{On the existence  of Pareto solutions}        

\author{Do Sang Kim$^1$ \and Ti\ees n-S\ow n Ph\d{a}m$^2$ \and Nguyen Van Tuyen$^3$}

\authorrunning{D. S. Kim, T. S. Ph\d{a}m, and N. V. Tuyen}

\institute{\Letter\ \ Nguyen Van Tuyen
	 \\
tuyensp2@yahoo.com; nguyenvantuyen83@hpu2.edu.vn
\\
\\
Do Sang Kim
\\
{dskim@pknu.ac.kr}
\\
\\
Ti\ees n-S\ow n Ph\d{a}m
\\
{sonpt@dlu.edu.vn}
\\
\\
			\at$^1$Department of Applied Mathematics, Pukyong National University, Busan 48513, Korea
	        \and
           \at$^2$Department of Mathematics, University of Dalat, 1 Phu Dong Thien Vuong, Dalat, Vietnam
           \and
           \at$^3$Department of Mathematics, Hanoi Pedagogical University 2, Xuan Hoa, Phuc Yen, Vinh Phuc, Vietnam  
}

\date{Received: date / Accepted: date}

\maketitle

\begin{abstract}
We are interested in the existence of Pareto solutions to the vector optimization problem 
$$\text{\rm Min}_{\,\mathbb{R}^m_+} \{f(x) \,|\, x\in \mathbb{R}^n\},$$
where $f\colon\mathbb{R}^n\to \mathbb{R}^m$ is a polynomial map. 
By using the {\em tangency variety} of $f$ we first construct a semi-algebraic set of dimension at most $m - 1$ containing the set of Pareto values of the problem. Then we establish connections between the Palais--Smale conditions, $M$-tameness, and properness for the map $f$. Based on these results, we provide some sufficient conditions for the existence of Pareto solutions of  the problem. We also introduce a generic class of polynomial vector optimization problems having at least one Pareto solution. 
\end{abstract}

\section{Introduction}\label{Introduction}

Existence of solutions and unboundedness are important issues in (vector) optimization theory; 
we refer the readers to the book \cite{jahn04} and to the papers \cite{bao07,bao10,borwein83,Gutierrez2014,ha06} with the references therein. In this paper, we are interested in the question about the existence of Pareto solutions to the unconstrained vector optimization problem 
\begin{equation}\label{problem}
\text{\rm Min}_{\,\mathbb{R}^m_+} \{f(x) \,|\, x\in \mathbb{R}^n\},   \tag{VP}
\end{equation}
where $f\colon\mathbb{R}^n\to \mathbb{R}^m$ is a polynomial map. 

We first consider the case $m = 1.$ It is well known that   \eqref{problem} has a solution if  the objective function $f$ is {\em coercive} on $\mathbb{R}^n,$ i.e., 
$f(x)\to +\infty$  when $\|x\|\to\infty.$ This condition is equivalent to the fact that $f$ is bounded from below and satisfies the so-called {\em Palais--Smale condition}; see the survey \cite{mauhin10} for more details. Regarding to the coercivity property of polynomials, see the recent papers \cite{Bajbar2015,Jeyakumar2014}.

We next assume that $m > 1.$ By introducing some variants of the Ekeland variational principle for set-valued maps, it was proved in \cite{bao07,bao10,ha06} that the set of weak Pareto solutions of \eqref{problem} is nonempty, provided that the following two conditions hold true:
\begin{itemize}
	\item[$\bullet$] $f$ is bounded from below, i.e., there exists an element $a \in \mathbb{R}^m$ such that
	$$f(\mathbb{R}^n)\subset a+\mathbb{R}^m_+.$$ 
	
	\item[$\bullet$] $f$ satisfies a Palais--Smale type condition.
\end{itemize}
Note that both of these assumptions seem to be  rather restrictive (see examples in Sections~\ref{Section3} and \ref{Section4} below). So we would like to find better sufficient conditions for the existence of Pareto solutions of \eqref{problem} in the case where $f$ is a polynomial map.
\medskip
\\
{\bf Contribution.} We study the existence of Pareto solutions in polynomial  vector optimization problems. To do this, we will use the so-called tangency varieties and tangency values at infinity.  It is worth noting that these concepts play important
roles in the study of polynomial optimization problems; see \cite{HaHV2017}. Namely, assume that the map $f$ is polynomial, then our contribution is as follows:

(a) We will construct a semi-algebraic subset of $\mathbb{R}^m$ of dimension at most $m - 1$ containing the set of Pareto values of \eqref{problem}. 
This subset can be estimated effectively as shown very recently in \cite{Dias2017}.

(b) Under the assumption that the image $f(\mathbb{R}^n)$ has a bounded section at some $\bar{t} \in \mathbb{R}^m,$ which is indeed necessary for the existence of Pareto solutions of  \eqref{problem}, we show that the following statements are equivalent:
\begin{itemize}
	\item[$\bullet$]  $f$ is proper at the sublevel $\bar t$.
	\item[$\bullet$] $f$ satisfies the Palais--Smale condition at the sublevel $\bar{t}.$
	\item[$\bullet$] $f$ satisfies the weak Palais--Smale condition at the sublevel $\bar{t}.$
	\item[$\bullet$] $f$ is $M$-tame at the sublevel $\bar{t}.$
\end{itemize}

(c) Based on these results, we provide some sufficient conditions under which the set of Pareto solutions of \eqref{problem} is nonempty. Finally, we show a generic class of vector optimization problems having at least one Pareto solution. 

We hope that the results in this paper will be useful in finding Pareto solutions/values of polynomial vector optimization problems.

To be concrete, we state the results for polynomial maps. Analogous results, with essentially identical proofs, hold for maps definable in an ``o-minimal structure" (such as semi-algebraic maps) or, even more generally, for ``tame" maps. See \cite{Dries1996} for more on the subject.

The rest of the paper is organized as follows. In Section~\ref{Section2} we recall some preliminary results from semi-algebraic geometry. Section~\ref{Section3} is devoted to Pareto values and tangencies. Some relationships between Palais--Smale conditions, $M$-tameness, and properness for polynomial maps are also established in this section. Several sufficient conditions for the existence of Pareto solutions of  \eqref{problem} are given in Section~\ref{Section4}. 
Section~\ref{Section5} draws some conclusions.

\section{Preliminaries} \label{Section2}

We use the following notation and terminology. Fix a number  $n \in {\Bbb N}$, $n \geq 1$, and abbreviate $(x_1, x_2, \ldots, x_n)$ by $x.$  The space $\mathbb{R}^n$ is equipped with the usual scalar product $\langle \cdot, \cdot \rangle$ and the corresponding Euclidean norm $\| \cdot\|.$ The interior (resp., the closure) of a set $S$ is denoted by $\mathrm{int}\, S$ (resp., $\mathrm{cl}\,{S}$). The  closed unit ball in $\mathbb{R}^n$ is denoted by $\mathbb{B}^n.$ Let $\mathbb{R}^m_+ := \{t := (t_1, \ldots, t_m)\,|\, \, t_i\geq 0,\,\, i=1,\ldots, m\}$ be the nonnegative orthant in $\mathbb{R}^m$. The cone $\mathbb{R}^m_+$  induces the following partial order in $\mathbb{R}^m$: $x, y\in \mathbb{R}^m$, $x\leq y$ if and only if $y - x\in\mathbb{R}^m_+.$

Now, we recall some notions and results of semi-algebraic geometry, which can be found in \cite{Bochnak1998,HaHV2017}.

\begin{definition}{\rm
\begin{enumerate}
			\item[(i)] A subset of $\mathbb{R}^n$ is called {\em semi-algebraic} if it is a finite union of sets of the form
			$$\{x \in \mathbb{R}^n \ | \ f_i(x) = 0, i = 1, \ldots, k; f_i(x) > 0, i = k + 1, \ldots, p\}$$
			where all $f_{i}$ are polynomials.
			\item[(ii)]
			Let $A \subset \Bbb{R}^n$ and $B \subset \Bbb{R}^m$ be semi-algebraic sets. A  map  $F \colon A \to B$ is said to be {\em semi-algebraic} if its graph
			$$\{(x, y) \in A \times B \ | \ y = F(x)\}$$
			is a semi-algebraic subset in $\Bbb{R}^n\times\Bbb{R}^m.$
\end{enumerate}
	}\end{definition}
	
By definition, it is easy to see that the class of semi-algebraic sets is closed under taking finite intersections, finite unions and complements; a Cartesian product of semi-algebraic sets is a semi-algebraic set. Furthermore, we have the following result (see  \cite[Proposition~2.2.7]{Bochnak1998} or \cite[Section~6]{HaHV2017}).
	
\begin{theorem}[Tarski--Seidenberg Theorem] \label{TarskiSeidenbergTheorem} 
		The  image  and  inverse image of  a  semi-algebraic set  under  a  semi-algebraic  map  are semi-algebraic sets.
\end{theorem}
	
\begin{remark}{\rm
As an immediate consequence of the Tarski--Seidenberg Theorem, we get semialgebraicity of any set $\{ x \in A \ | \ \exists y \in B,  (x, y) \in C \},$  provided that $A ,  B,$  and $C$  are semi-algebraic sets in the corresponding spaces. It also follows that  $\{ x \in A \ | \ \forall y \in B,  (x, y) \in C \}$ is a semi-algebraic set as its complement is the union of the complement of $A$  and the set $\{ x \in A \ | \ \exists y \in B,  (x, y) \not\in C \}.$ Thus, if we have a finite collection of semi-algebraic sets, then any set obtained from them with the help of a finite chain of quantifiers is also semi-algebraic. In particular, it is not hard to see that the closure and the interior of a semi-algebraic set are semi-algebraic sets.
}\end{remark}
		
By the Cell Decomposition Theorem (see \cite[Theorem~2.3.6]{Bochnak1998}), for any $p \in \mathbb{N}$ and any nonempty semi-algebraic subset $A \subset \mathbb{R}^n,$ we can write $A$ as a disjoint union of finitely many semi-algebraic $C^p$-manifolds of different dimensions. The {\em dimension} $\dim A$ of a nonempty semi-algebraic set $A$ can thus be defined as the dimension of the manifold of highest dimension of its decomposition. This dimension is well defined and independent of the decomposition of $A.$ By convention, the dimension of the empty set is taken to be negative infinity. We will need the following result (see \cite{Bochnak1998,HaHV2017}).
		
\begin{proposition} \label{DimensionProposition}
\begin{enumerate}
				\item [{\rm (i)}] Let $A \subset \mathbb{R}^n$ be a semi-algebraic set and
				$f \colon A \to\mathbb{R}^m$ a semi-algebraic map. Then
				$\dim f(A)\leq\dim  A.$
				
				\item [{\rm (ii)}] Let $A \subset \mathbb{R}^n$ be a nonempty semi-algebraic set. Then
				$$\dim(\mathrm{cl}\,{A}\setminus A) < \dim A.$$
				In particular, $\dim\mathrm{cl}\,{A}=\dim A.$
				
				\item [{\rm (iii)}] Let $A, B \subset \mathbb{R}^n$ be semi-algebraic sets. Then 
				$$\dim (A \cup B) = \max\{ \dim A, \dim B\}.$$
\end{enumerate}
\end{proposition}
		
In the sequel, we will need the following useful results (see, for example, \cite{HaHV2017}).
		
\begin{lemma}[Curve Selection Lemma at infinity]\label{CurveSelectionLemma}
Let $A\subset \mathbb{R}^n$ be a semi-algebraic set, and let $f := (f_1, \ldots,f_m) \colon  \mathbb{R}^n \to \mathbb{R}^m$ be a semi-algebraic map. Assume that there exists a sequence $\{x^\ell\}$ such that $x^\ell \in A$, $\lim_{\ell \to \infty} \| x^\ell  \| = \infty$ and $\lim_{\ell \to \infty} f(x^\ell)  = y \in(\overline{\mathbb{R}})^m,$ where $\overline{\mathbb{R}} := \mathbb{R} \cup \{\pm \infty\}.$ Then there exists a smooth semi-algebraic curve $\varphi \colon (0, \epsilon)\to \mathbb{R}^n$ such that $\varphi(t) \in A$ for all $t \in (0, \epsilon), \lim_{t \to 0} \|\varphi(t)\| = \infty,$ and $\lim_{t \to 0} f(\varphi(t)) = y.$
\end{lemma}
		
\begin{lemma}[Growth Dichotomy Lemma] \label{GrowthDichotomyLemma}
Let $f \colon (0, \epsilon) \rightarrow {\Bbb R}$ be a semi-algebraic function with $f(t) \ne 0$ for all $t \in (0, \epsilon).$ Then there exist constants $c \ne 0$ and $q \in {\Bbb Q}$ such that $f(t) = ct^q + o(t^q)$ as $t \to 0^+.$
\end{lemma}
		
\begin{lemma}[Monotonicity Lemma] \label{MonotonicityLemma}
Let $a < b$ in $\mathbb{R}.$ If $f \colon [a, b] \rightarrow \mathbb{R}$ is a semi-algebraic function, then there is a partition $a =: t_1 < \cdots < t_{N} := b$ of $[a, b]$ such that $f|_{(t_l, t_{l + 1})}$ is $C^1,$ and either constant or strictly monotone, for $l \in \{1, \ldots, N - 1\}.$
\end{lemma}
		
\section{Pareto values and tangencies}\label{Section3}
		
\subsection{Pareto values}
		
Let $f := (f_1,\ldots, f_m)\colon\mathbb{R}^n\to\mathbb{R}^m$ be a map and consider the vector optimization problem~\eqref{problem} formulated in Section~\ref{Introduction}.
		
\begin{definition}{\rm 
Let $t\in \,\mathrm{cl} f(\mathbb{R}^n)$. We say that:
\begin{enumerate}
	\item [{\rm (i)}] $t$ is a {\em Pareto (optimal) value} of ~\eqref{problem} if 
				$$f(x)\notin t - (\mathbb{R}^m_+\setminus\{0\}) \quad \textrm{ for all } \quad x \in \mathbb{R}^n.$$
				The set of  all Pareto values of \eqref{problem}  is denoted by $\mathrm{val}\,\eqref{problem}$.
				
\item [{\rm (ii)}] $t$ is a {\em weak Pareto (optimal) value} of \eqref{problem} if 
				$$f(x)\notin t -{\rm int}\,\mathbb{R}^m_+\quad \textrm{ for all } \quad  x\in \mathbb{R}^n.$$
				The set of  all weak Pareto values of \eqref{problem} is denoted by $\mathrm{val}^w\,\eqref{problem}$.
				
\item [{\rm (iii)}] A point $x^*$ is said to be a {\em Pareto (optimal) solution} (resp., {\em weak Pareto (optimal) solution})  if $f(x^*)$ is a Pareto value (resp., weak Pareto value) of \eqref{problem}. The set of all Pareto solutions (resp., weak Pareto solutions)  is denoted by $\mathrm{sol}\,\eqref{problem}$ (resp., $\mathrm{sol}^w\,\eqref{problem}$).
\end{enumerate}
}\end{definition}
			
\begin{remark}{\rm 
\begin{enumerate}
	\item [{\rm (i)}] By definition, it is clear that $\mathrm{val}\,\eqref{problem}\subset \mathrm{val}^w\,\eqref{problem}.$ Note that the inclusion may be strict. 
					
	\item [{\rm (ii)}] In the case of $m = 1$ and $f$ is bounded from below on $\mathbb{R}^n$, 
					$$\mathrm{val}\,\eqref{problem}=\mathrm{val}^w\,\eqref{problem}=\{\inf_{x\in\mathbb{R}^n} f(x)\}.$$
					
	\item [{\rm (iii)}] A (weak) Pareto value of the problem~\eqref{problem} does not necessarily belong to $f(\mathbb{R}^n)$ as shown in the example below.
\end{enumerate}
}\end{remark}
				
\begin{example}{\rm 
\begin{enumerate}
	\item [{\rm (i)}]	Let $f\colon\mathbb{R}^3\to\mathbb{R}^2$ be the polynomial map defined by
$$f(x_1, x_2, x_3):=(x_3, x_1^2+(x_1x_2-1)^2+x_3^2).$$ 
We have that
$$f(\mathbb{R}^3)=\{t=(t_1, t_2)\in\mathbb{R}^2\,|\, t_2>t_1^2\}$$
is an open set in $\mathbb{R}^2$. Furthermore, it is easy to see that 
$$\mathrm{val}\,\eqref{problem}=\mathrm{val}^w\,\eqref{problem}=\{t=(t_1, t_2)\in\mathbb{R}^2\,|\, t_2=t_1^2, t_1\leq 0\} \ne \emptyset.$$
Hence $\mathrm{val}\,\eqref{problem}\cap f(\mathbb{R}^3) = \mathrm{val}^w\,\eqref{problem}\cap f(\mathbb{R}^3) = \emptyset,$ and so $\mathrm{sol}\,\eqref{problem}=\mathrm{sol}^w\,\eqref{problem}=\emptyset.$
						
\item [{\rm (ii)}] In the recent paper \cite{Fernando2016} (see also \cite{Fernando2003,Fernando2017}) it was proved that the open quadrant
$$\{(t_1, t_2) \in \mathbb{R}^2 \, | \, t_1 > 0, t_2 > 0\}$$
is the image of the polynomial map
$f \colon \mathbb{R}^2 \to \mathbb{R}^2, \ (x_1,x_2) \mapsto ((x_1^2x_2^4 + x_1^4x_2^2 - x_2^2 -1)^2 + x_1^6x_2^4, (x_1^6x_2^2 + x_1^2x_2^2 - x_1^2 - 1)^2 + x_1^6x_2^4).$
For this $f,$ we have
$$\mathrm{val}\,\eqref{problem}=\mathrm{val}^w\,\eqref{problem} = \{(t_1, t_2) \in \mathbb{R}^2 \, | \, t_1 t_2 = 0, t_1 \ge 0, t_2 \ge 0 \}  \ne \emptyset.$$
Therefore, $\mathrm{val}\,\eqref{problem}\cap f(\mathbb{R}^2) = \mathrm{val}^w\,\eqref{problem}\cap f(\mathbb{R}^2) = \emptyset,$ and so $\mathrm{sol}\,\eqref{problem}=\mathrm{sol}^w\,\eqref{problem}=\emptyset.$
\end{enumerate}
}\end{example}
					
\begin{remark}{\rm  
It was proved very recently in \cite{huong16} that both the proper Pareto solution set and the weak Pareto solution set of a vector variational inequality, where the convex constraint set is given by polynomial functions and all the components of the basic operators are polynomial functions, have finitely many connected components, provided that the Mangasarian--Fromovitz constraint qualification is satisfied at every point of the constraint set. In addition, if the proper Pareto solution set is dense in the Pareto solution set, then the latter also has finitely many connected components. Applying the above result to vector optimization problems under polynomial constraints, where all the components of the basic operators are polynomial functions, the authors obtained some topological properties of the stationary point set, as well as the weak Pareto solution set, of the problem in question.
							
We would like to remark that all the results in the cited paper can be concluded immediately from Theorem~\ref{TarskiSeidenbergTheorem} without any convexity assumption or constraint qualification conditions. Indeed it suffices to assume that maps and constraint sets are semi-algebraic. As an illustrative example, we prove here that the sets $\mathrm{val}\,\eqref{problem}$  and $\mathrm{sol}\,\eqref{problem}$ are semi-algebraic provided that $f$ is a (not necessarily continuous) semi-algebraic map and so, thanks to \cite[Theorem~2.4.4]{Bochnak1998}, they have a finite number of (path) connected components. 
							
Let $f :=(f_1,\ldots, f_m)\colon\mathbb{R}^n\to\mathbb{R}^m$ be a semi-algebraic map. By Theorem~\ref{TarskiSeidenbergTheorem}, the set $f(\mathbb{R}^n)$ is semi-algebraic and so is $\mathrm{cl}\,f(\mathbb{R}^n).$ Let $\phi$ and $\psi$ be two functions defined by
\begin{align*}
\phi &\colon \mathbb{R}^n\times\mathbb{R}^m\to\mathbb{R}, \quad (x, t)\mapsto \max_{i}\{f_i(x)-t_i\},\\
\psi&\colon \mathbb{R}^n\times\mathbb{R}^m\to\mathbb{R}, \quad (x, t)\mapsto \sum_{i=1}^m [f_i(x)-t_i]^2.
\end{align*}
In view of Theorem~\ref{TarskiSeidenbergTheorem}, it is easy to see that $\phi $ and $\psi$ are semi-algebraic functions. Furthermore, by definition we have
\begin{align*}
\mathrm{val}\,\eqref{problem}&=\{t\in \mathrm{cl}\, f(\mathbb{R}^n)\,\,|\,\, \forall x\in\mathbb{R}^n, f(x) \notin t - (\mathbb{R}^m_+\setminus\{0\})\}
							\\
&=\{t\in \mathrm{cl}\, f(\mathbb{R}^n)\,\,|\,\,\forall x\in\mathbb{R}^n, \phi (x, t)>0 \,\,\mbox{\rm or}\,\, \psi(x, t)=0\}.
\end{align*}
Note that $\psi(x, t)\geq 0$ on $\mathbb{R}^n\times\mathbb{R}^m$. Hence
$$
\mathrm{cl}\,f(\mathbb{R}^n)\setminus\mathrm{val}\,\eqref{problem}= \{t\in \mathrm{cl}\, f(\mathbb{R}^n)\,\,|\,\, \exists x\in\mathbb{R}^n,  \phi (x, t)\leq 0 \,\,\mbox{\rm and}\,\, \psi(x, t)>0\}.
$$
Thanks to Theorem~\ref{TarskiSeidenbergTheorem}, this set is   semi-algebraic because it is the projection onto the last $m$ coordinates of the following semi-algebraic set 
$$\{(x, t)\in\mathbb{R}^n \times \mathrm{cl}\, f(\mathbb{R}^n) \,\,|\,\, \phi (x, t)\leq 0 \,\,\mbox{\rm and}\,\, \psi(x, t)>0\}.$$ 
Therefore, $\mathrm{val}\eqref{problem}$ is   a semi-algebraic set. 

Finally, the set $\mathrm{sol}\,\eqref{problem} = f^{-1}(\mathrm{val}\,\eqref{problem})$
is semi-algebraic because of Theorem~\ref{TarskiSeidenbergTheorem} again.
							
Similarly, it is easy to check that the sets $\mathrm{val}^w\,\eqref{problem}$ and $\mathrm{sol}^w\,\eqref{problem}$ are semi-algebraic and so, by \cite[Theorem~2.4.4]{Bochnak1998}, they have a finite number of connected components, which are semi-algebraic. 
}\end{remark}
						
\subsection{Tangencies} 
Let $f:=(f_1,\ldots, f_m)\colon\mathbb{R}^n\to\mathbb{R}^m$ be a polynomial  map. A point $t\in\mathbb{R}^m$ is called a {\em regular value} for $f$ if either $f^{-1}(t)=\emptyset$ or the derivative map $D f(x)\colon\mathbb{R}^n\to\mathbb{R}^m$ is surjective at every point $x\in f^{-1}(t)$.  A point $t\in\mathbb{R}^m$ that is not a regular value of $f$
is called a {\em critical value}. We will denote by $K_0(f)$ the set of critical values of $f$. 
						
\begin{definition}[see \cite{HaHV2017}]{\rm 
(i) By the {\em tangency variety of $f$} we mean the set
$$\Gamma (f):=\{x\in \mathbb{R}^n\,|\, \exists \lambda_i,  \mu\in\mathbb{R},  \,\,\mbox{not all zero, such that}  \,\, \sum_{i=1}^m \lambda_i\nabla f_i(x)+ \mu x=0\},$$
here and in the following $\nabla f_i(x)$ stands for the gradient of $f_i$ at $x.$
								
(ii) The set of {\em tangency values {\rm(}at infinity{\rm)}} of $f$ is defined by
$$T_{\infty}(f):=\{t\in\mathbb{R}^m\,\,|\,\,  \exists \{x^k\}\subset \Gamma(f), \|x^k\|\to +\infty \,\,\mbox{and}\,\, f(x^k)\to t \,\,\mbox{as}\,\, k\to\infty\}.$$
}\end{definition}
							
\begin{remark}{\rm	Very recently, relying on results from semi-algebraic geometry, it was proved in \cite{Dias2017} (see also \cite{Dias2015,Jelonek2014}) that the set of tangency values at infinity of polynomial maps can be estimated effectively.
}\end{remark}
								
\begin{lemma}\label{Lemma31}
$\Gamma(f)$ is an unbounded nonempty semi-algebraic set.
\end{lemma}
\begin{proof}
By Theorem~\ref{TarskiSeidenbergTheorem}, it is easy to check that the set $\Gamma(f)$ is semi-algebraic.
									
We next show that $\Gamma(f) \ne \emptyset.$ To this end, take any $R > 0.$ Then the sphere $S_R := \{x \in \mathbb{R}^n \ | \ \|x\|^2 = R^2 \}$ is nonempty compact. Hence, the optimization problem 
$\text{\rm Min}_{\,\mathbb{R}^m_+} \{f(x) \,|\, x\in S_R\}$	has a Pareto solution, say $x(R) \in S_R.$ The Fritz-John optimality conditions \cite[Theorem~7.4]{jahn04} imply that $x(R) \in \Gamma(f),$ and so $\Gamma(f) \ne \emptyset.$  Finally, it is clear that if $R \to \infty$ then $\|x(R)\| = R \to \infty,$ which proves the lemma. $\hfill\Box$
\end{proof}
								
We now give a simple and constructive proof of the following known result \cite[Theorem~2.5]{Dias2015}, \cite[Theorem~5.7]{Dias2012}, \cite[Theorem~1.1]{HaHV2017} and \cite[Theorem~1.5]{Loi1998}.
\begin{proposition}\label{Proposition31}
	$T_{\infty}(f)$ is a closed semi-algebraic set of dimension at most $m - 1.$
\end{proposition}
\begin{proof}	By definition and Theorem~\ref{TarskiSeidenbergTheorem}, it is not hard to check that $T_{\infty} (f)$ is a closed semi-algebraic set.
									
Consider the semi-algebraic map 
$$\Phi \colon \mathbb{R}^n \rightarrow \mathbb{R}^{m + 1}, \quad x \mapsto (f(x), \|x\|^2) .$$ 
In view of Lemma~\ref{Lemma31} and Theorem~\ref{TarskiSeidenbergTheorem}, the image $\Phi(\Gamma(f))$ is semi-algebraic. By definition, $\Gamma(f)$ is the set of critical points of $\Phi.$ Thanks to Sard's theorem (see, for example, \cite[Theorem 1.9]{HaHV2017}), the set $\Phi(\Gamma(f))$ is of dimension at most $m,$  and so it cannot contain a nonempty open subset of ${\mathbb R}^{m + 1}.$ 
									
On the other hand, since $\Phi(\Gamma(f))$ is semi-algebraic, we can write 
$$\Phi(\Gamma(f)) = \bigcup_{i = 1}^s \{(t, R) \in {\mathbb R}^m \times {\mathbb R}  \, | \, g_i(t, R) = 0,\  h_{i_j}(t, R) > 0, \ j = 1, \ldots, k_i\}$$
for some polynomials $g_i$ and $h_{i_j}.$ Then we must have $g_i \not \equiv 0$ for all $i = 1, \ldots, s,$ because otherwise $\Phi(\Gamma(f))$ would contain a nonempty open subset of ${\mathbb R}^{m + 1}$, a contradiction. Let $P \colon \mathbb{R}^{m + 1} \rightarrow \mathbb{R}$ be the product of all the polynomials $g_i, i = 1, \ldots, s.$ Clearly, $P \not \equiv 0$ and 
$$\Phi(\Gamma(f))  \subset \{(t, R) \in \mathbb{R}^m \times \mathbb{R} \ | \ P(t, R) = 0 \}.$$
Write
$$P(t, R) = a_0(t) R^d + \cdots + a_d(t)$$
for some polynomials $a_i(t)$ with $a_0(t) \not \equiv 0.$ By definition, then
\begin{eqnarray*}
T_{\infty}(f)
	&=& \{t\in\mathbb{R}^m \ | \ \exists (t^k, R^k) \in \Phi(\Gamma(f)), R^k \to +\infty \,\,\mbox{and}\,\, t^k \to t \,\,\mbox{as}\,\, k\to\infty\} \\
	&\subset& \{t\in\mathbb{R}^m \ | \ \exists (t^k, R^k) \in P^{-1}(0), R^k \to +\infty \,\,\mbox{and}\,\, t^k \to t \,\,\mbox{as}\,\, k\to\infty\} \\
		&\subset& \{t \in\mathbb{R}^m \ | \ a_0(t) = 0 \}.
\end{eqnarray*}
Therefore, $\dim T_\infty(f) \le m - 1,$ which completes the proof.  $\hfill\Box$
\end{proof}
								
\begin{remark}{\rm	In \cite{Magron2014,Magron2015}, by using semidefinite programming relaxations, the authors provided several methods to approximate as closely as desired the image of semi-algebraic sets under polynomial maps with super-level sets of single polynomials of fixed degrees. This fact, together with the proof of Proposition~\ref{Proposition31}, gives us a hope that the set $\Phi(\Gamma(f))$  and so $T_\infty(f)$ can be approximated effectively.
}\end{remark}
									
The next statement describes a relation between Pareto values and tangency values.
									
\begin{theorem} \label{Theorem31}The following inclusions hold true
$$\mathrm{val}\,\eqref{problem} \subset \mathrm{val}^w\,\eqref{problem}\subset K_0(f)\cup T_{\infty}(f).$$ 
In particular, the semi-algebraic sets $\mathrm{val}\,\eqref{problem}$  and $\mathrm{val}^w\,\eqref{problem}$ are of dimension at most $m - 1.$
\end{theorem}
									
\begin{proof}The first inclusion is obvious. Let us prove the second one. Fix $t\in \mathrm{val}^w\, \eqref{problem}$. 
										
If $t\in f(\mathbb{R}^n)$, then $t\in K_0(f)$ due to the Karush--Kuhn--Tucker necessary conditions \cite[Theorem~7.4]{jahn04}. So assume that $t \in\,\mathrm{cl} f(\mathbb{R}^n)\setminus f(\mathbb{R}^n)$. Then there is a sequence $\{x^k\}$ such that $\displaystyle\lim_{k\to\infty} f(x^k)=t$. We claim that $\displaystyle\lim_{k\to\infty}\|x^k\|=+\infty$. Indeed, if it is not the case, then the sequence $\{x^k\}$  has an accumulation point, say $x^* \in \mathbb{R}^n.$  By the continuity of $f$, we have $t = f(x^*) \in f(\mathbb{R}^n),$ which is a contradiction. 
										
For each $k\in\mathbb{N}$, we  consider the scalar optimization problem
\begin{eqnarray*}
&& \min \|f(x)-t\|^2 \\
&& \textrm{s.t.} \quad x \in \mathbb{R}^n, \ \|x\|^2=\|x^k\|^2.
\end{eqnarray*}
Since $\{x\in\mathbb{R}^n\,\,|\,\,  \|x\|^2=\|x^k\|^2\}$ is a nonempty compact set in $\mathbb{R}^n$, this problem admits an optimal solution, say $y^k$. It is easy to check that the sequence $\{y^k\}$ has the following properties:
\begin{enumerate}
\item[(a)] $\displaystyle\lim_{k\to\infty} \|y^k\|= \lim_{k\to\infty} \|x^k\|=+\infty$,
											
\item[(b)] $0\leq \|f(y^k)-t\|^2\leq \|f(x^k)-t\|^2$, and 
											
\item[(c)] there exists $\mu^k\in\mathbb{R}$ such that
$$\sum_{i=1}^m (f_i(y^k)-t_i)\nabla f_i(y^k) +\mu^k y^k=0.$$
(This follows from the Karush--Kuhn--Tucker necessary conditions.)
\end{enumerate}
Since $t\notin f(\mathbb{R}^n)$, one has $f(y^k)\neq t$ for all $k\in\mathbb{N}$. Therefore $\{y^k\}\subset \Gamma (f)$. Moreover, we have
$$0\leq\lim_{k\to\infty}\|f(y^k)-t\|^2\leq \lim_{k\to\infty}\|f(x^k)-t\|^2=0,$$ 
and so $\displaystyle\lim_{k\to\infty}f(y^k) = t.$ Thus  $t\in T_{\infty} (f).$ 
										
Finally, due to the Sard theorem (see, for example,  \cite[Theorem 1.9]{HaHV2017}), $K_0(f)$ is a semi-algebraic set of dimension at most $m - 1.$ 
This, together with Propositions~\ref{DimensionProposition} and \ref{Proposition31}, implies the last statement.  $\hfill\Box$
\end{proof}
									
\subsection{Palais--Smale conditions, $M$-tameness and properness}
									
Given a differentiable map $f := (f_1,\ldots, f_m)\colon\mathbb{R}^n\to\mathbb{R}^m$ and a value $\bar{t} \in (\mathbb{R} \cup \{+ \infty\})^m,$ we let
\begin{eqnarray*}
\widetilde{K}_{\infty, \le \bar{t}}(f) &:=& \{t\in\mathbb{R}^m\,|\,  \exists \{x^k\}\subset \mathbb{R}^n, f(x^k) \le \bar t, \|x^k\|\to +\infty,  f(x^k)\to t,  \mbox{and } \\
& & \qquad \qquad \qquad \qquad \qquad \qquad\qquad \qquad\ \  \nu_f(x^k)\to 0  \,\mbox{ as } \, k\to\infty\},\\
K_{\infty, \le {\bar{t}}}(f) &:= & \{t\in\mathbb{R}^m\,|\,  \exists \{x^k\}\subset \mathbb{R}^n, f(x^k) \le {\bar{t}}, \|x^k\|\to +\infty, f(x^k)\to t, \mbox{and } \\
& & \qquad \qquad \qquad \qquad \qquad \qquad\qquad\, \   \|x^k\|\nu_f(x^k)\to 0  \, \mbox{ as } \, k\to\infty\}, \\
T_{\infty, \le {\bar{t}}}(f) &:=& \{t\in\mathbb{R}^m\,\,|\,\,  \exists \{x^k\}\subset \Gamma(f), f(x^k) \le {\bar{t}}, \|x^k\|\to +\infty,  \,\mbox{and}
\\
& & \qquad \qquad \qquad \qquad \qquad \qquad\qquad \qquad\ \ \ \   f(x^k)\to t \,\,\mbox{as}\,\, k\to\infty\},
\end{eqnarray*}
where $\nu_f\colon\mathbb{R}^n\to\mathbb{R}$ is the Rabier function (see \cite{kurdyka00,Rabier1997}) defined by
$$\nu_f(x):=\min_{\sum_{i=1}^m|\lambda_i|=1}\left\|\sum_{i=1}^m \lambda_i\nabla f_i(x)\right\|.$$
Note that if $m = 1$ then $\nu_f(x) = \| \nabla f(x)\|.$

For simplicity of notation, when $\bar{t} = (+\infty, \ldots, +\infty),$ we write $\widetilde{K}_{\infty}(f),$ ${K}_{\infty}(f),$ and $T_\infty(f)$ instead of 
$\widetilde{K}_{\infty, \le \bar{t}}(f),$ $K_{\infty, \le \bar{t}}(f),$ and $T_{\infty, \le \bar{t}}(f),$ respectively.
									
The following result is a generalization of \cite[Theorem 2.8]{Dias2015}, \cite[Theorem~1.1]{vui07},  and \cite[Proposition~3.1]{kurdyka00}.
									
\begin{proposition} \label{Proposition32}
Let $f\colon \mathbb{R}^n\to\mathbb{R}^m$ be a polynomial map and $\bar{t} \in (\mathbb{R} \cup \{+ \infty\})^m.$ The following inclusions hold:
\begin{equation*} T_{\infty, \le {\bar{t}}}(f)\subset K_{\infty, \le {\bar{t}}}(f)\subset \widetilde{K}_{\infty, \le \bar{t}}(f).
\end{equation*}
Furthermore, if $n \le m,$ then  these inclusions are equalities.
\end{proposition}
\begin{proof}
The second inclusion is immediate from the definitions.
										
To prove the first inclusion, take any $t \in T_{\infty, \le \bar{t}} (f).$ By definition, there exist sequences $\{x^k\} \subset \mathbb{R}^n$ and $\{(\lambda^k, \mu^k)\} \subset \left(\mathbb{R}^m \times  \mathbb{R}\right) \setminus \{0\}$ such that
\begin{eqnarray*}
&&\lim_{k \to \infty}  \|x^k\| = +\infty,  \quad \lim_{k \to \infty}  f(x^k) = t, \quad f(x^k) \le \bar{t}, \quad \sum_{i = 1}^m \lambda^k_i \nabla f_i(x^k) + \mu^k x^k = 0
\end{eqnarray*}
We can assume, after a scaling if necessary, that $\|(\lambda^k, \mu^k)\| = 1$ for all $k\in\mathbb{N}$.
										
Let
\begin{eqnarray*}
\mathscr{A}:=\big\{ (x, \lambda, \mu) \in \mathbb{R}^n \times \mathbb{R}^{m} \times \mathbb{R}
| f(x) \le \bar{t}, \sum_{i = 1}^m \lambda_i \nabla f_i(x) + \mu x = 0, \|(\lambda, \mu)\| = 1 \big\}.
\end{eqnarray*}
Then $\mathscr{A}$ is a semi-algebraic set and the sequence $(x^k, \lambda^k, \mu^k) \in \mathscr{A}$ tends to infinity as $k \to \infty.$ By applying Lemma~\ref{CurveSelectionLemma} for the semi-algebraic map $\mathscr{A} \rightarrow \mathbb{R}^m, (x, \lambda, \mu)  \mapsto f(x),$ 
we get a smooth semi-algebraic curve 
$$(\varphi, \lambda, \mu) \colon (0, \epsilon) \rightarrow \mathbb{R}^n \times \mathbb{R}^{m} \times \mathbb{R}, \quad \tau \mapsto (\varphi(\tau), \lambda(\tau), \mu(\tau)),$$ 
satisfying the following conditions
\begin{enumerate}
\item [{(a)}] $\lim_{\tau \to 0^+}  \|\varphi(\tau)\|  = +\infty;$ 
\item [{(b)}] $\lim_{\tau \to 0^+}  f(\varphi(\tau)) =   t;$
\item [{(c)}] $f(\varphi(\tau)) \le \bar{t};$
\item [{(d)}] $\sum_{i = 1}^m \lambda_i(\tau) \nabla f_i(\varphi(\tau)) + \mu(\tau) \varphi(\tau) \equiv 0;$
\item [{(e)}] $\|(\lambda(\tau), \mu(\tau))\| \equiv 1.$
\end{enumerate}
										
Since the (smooth) functions $\lambda_i, \mu,$ and $f_i \circ \varphi$ are semi-algebraic, we can assume, by shrinking $\epsilon$ if necessary, that these functions are either constant or strictly monotone (see Lemma~\ref{MonotonicityLemma}). 
										
It follows from~(d) that
\begin{eqnarray*}
\frac{\mu (\tau)}{2} \frac{d \|\varphi(\tau)\|^2}{d\tau}
&=& \mu (\tau) \left \langle \varphi(\tau), \frac{d \varphi(\tau)}{d\tau} \right \rangle \\
&=& - \sum_{i = 1}^m \lambda_i(\tau) \left \langle \nabla  f_i(\varphi(\tau)), \frac{d \varphi(\tau)}{d\tau} \right \rangle \\
&=& - \sum_{i = 1}^m \lambda_i(\tau) \frac{d}{d\tau}(f_i \circ \varphi)(\tau).
\end{eqnarray*}
Let $I := \{i \in \{1, \ldots, m\} \ | \  \lambda_i (\tau) \frac{d}{d \tau} (f_i \circ \varphi )(\tau) \not \equiv 0\}.$ Then
\begin{eqnarray}\label{EqnBS17}
\frac{\mu (\tau)}{2} \frac{d \|\varphi(\tau)\|^2}{d\tau}
	&=& - \sum_{i \in I} \lambda_i(\tau) \frac{d}{d\tau}(f_i \circ \varphi)(\tau).
\end{eqnarray}
										
Assume that $I = \emptyset.$ From (a) and \eqref{EqnBS17} we have $\mu(\tau) \equiv 0.$ By (d), hence $\nu_f(\varphi(\tau)) \equiv 0,$ which together with (a)-(c), yields $t \in K_{\infty, \le \bar{t}}(f).$ 
										
We now assume that $I \ne \emptyset.$  For each $i \in I,$ we have $\lambda_i(\tau)  \not \equiv 0 $ and $f_i \circ \varphi (\tau)  \not \equiv t_i.$ By Lemma~\ref{GrowthDichotomyLemma}, we may write
\begin{eqnarray*}
\lambda_i(\tau) & = & a_i \tau^{\alpha_i} + \textrm{higher order terms in } \tau,\\
f_i \circ \varphi (\tau) & = & t_i + b_i \tau^{\beta_i} + \textrm{higher order terms in } \tau,
\end{eqnarray*}
where $a_i \ne 0, b_i \ne 0$ and $\alpha_i, \beta_i \in \mathbb{Q}.$ By Conditions (e) and (b) respectively, we have $\alpha_i \ge 0$ and $\beta_i > 0.$ In particular, $\theta := \min_{i \in I} (\alpha_i + \beta_i) > 0.$ 
										
On the other hand, from (d) and \eqref{EqnBS17}, we have
\begin{eqnarray*}
\frac{\left \| \displaystyle \sum_{i = 1}^m \lambda_i(\tau) \nabla f_i(\varphi(\tau))  \right  \|}{2\| \varphi(\tau)\|} \left | \frac{d \|\varphi(\tau)\|^2}{d \tau} \right |
	&=& \left | \sum_{i \in I} \lambda_i(\tau) \frac{d}{d\tau}(f_i \circ \varphi)(\tau) \right |.
\end{eqnarray*}
Note that asymptotically as $\tau \to 0^+,$ 
\begin{eqnarray*}
\|\varphi(\tau)\|^2 & \simeq & \tau \frac{d \|\varphi(\tau)\|^2}{d\tau} .
\end{eqnarray*}
Therefore, 
\begin{eqnarray*}
\|\varphi(\tau)\| \left\|\sum_{i = 1}^m \lambda_i(\tau) \nabla f_i(\varphi(\tau)) \right\| 
	&\simeq& \frac{\left \| \displaystyle \sum_{i = 1}^m \lambda_i(\tau) \nabla f_i(\varphi(\tau))  \right  \|}{2\| \varphi(\tau)\|} \left | \tau  \frac{d \|\varphi(\tau)\|^2}{d \tau} \right | \\
&=& \left | \sum_{i \in I} \lambda_i(\tau) \tau  \frac{d}{d\tau}(f_i \circ \varphi)(\tau) \right | \\
&=& c \tau^{\theta}   + \textrm{higher order terms in } \tau,
\end{eqnarray*}
for some constant $c \ge 0.$ Since $\theta > 0,$ we have
\begin{eqnarray*}
\lim_{t \to 0^+} \|\varphi(\tau)\| \left  \|\sum_{i = 1}^m  \lambda_i(\tau) \nabla f_i(\varphi(\tau)) \right\| &=& 0.
\end{eqnarray*}
Combining this with (a)-(c) one gets $t \in K_{\infty, \le \bar{t}}(f),$ thus ending the proof of the first part of our statement.
										
We now assume that $n \le m.$ By definition, $\Gamma(f) = \mathbb{R}^n,$ and so $$T_{\infty, \le {\bar{t}}}(f) \supset \widetilde{K}_{\infty, \le \bar{t}}(f).$$ 
This, together with proven inclusions, gives the following equalities:
\begin{equation*} 
T_{\infty, \le {\bar{t}}}(f)  = K_{\infty, \le {\bar{t}}}(f) = \widetilde{K}_{\infty, \le \bar{t}}(f).
\end{equation*}
 $\hfill\Box$
\end{proof}

\begin{remark}{\rm
(i) The first inclusion in Proposition~\ref{Proposition32} may be strict. For example, consider a class of polynomial functions defined by 
$$f_{nq}\colon \mathbb{R}^3\to\mathbb{R}, \quad (x_1, x_2, x_3)\mapsto x_1-3x_1^{2n+1}x_2^{2q}+2x_1^{3n+1}x_2^{3q}+x_2x_3,$$
where $n, q\in\mathbb{N}\setminus\{0\}$. By a similar argument as in \cite{pau97}, we can show that $T_{\infty}(f_{nq}) = \emptyset$ and that $K_{\infty}(f_{nq}) = \emptyset$ if, and only if, $n \leq q$. For $n>q$ we therefore get $T_{\infty}(f_{nq}) \varsubsetneq K_{\infty}(f_{nq}) \ne \emptyset.$

(ii) According to \cite[Lemma 3.5]{kurdyka00} (see also \cite[Theorem~2]{Ioffe2007} and \cite[Theorem~6.4]{Jelonek2002}), we have 
$$\dim K_{\infty, \le {\bar{t}}}(f) \le \dim K_{\infty}(f) < m.$$
											
On the other hand, without some extra hypothesis the set $\widetilde{K}_{\infty, \le \bar{t}}(f)$ may be quite large in the sense that $\dim \widetilde{K}_{\infty, \le \bar{t}}(f) = m.$ For example, let $f\colon\mathbb{R}^3\to\mathbb{R}$ be the polynomial defined by  $f(x_1, x_2, x_3):= x_1+x_1^2x_2+x_1^4x_2x_3.$ Then it is not hard to check that $\widetilde{K}_{\infty}(f)=\mathbb{R}$ (see \cite[Example~2.1]{kurdyka00}), and hence   
	$$\dim {T}_{\infty}(f)=0<1=\dim\widetilde{K}_{\infty}(f).$$ 
}\end{remark}

\begin{definition}{\rm 
Let $A$ be a subset in $\mathbb{R}^m$ and ${\bar{t}} \in \mathbb{R}^m$. The set $A \cap ({\bar{t}} - \mathbb{R}^m_+)$ is called a {\em section} of $A$ at ${\bar{t}}$ and denoted by $[A]_{\bar{t}}.$ The section $[A]_{\bar{t}}$ is said to be {\em bounded} if, and only if, there is $a\in\mathbb{R}^m$ such that
$$[A]_{\bar{t}} \subset a + \mathbb{R}^m_+.$$ 
}\end{definition}
											
\begin{remark}{\rm 
\begin{enumerate}
\item [{\rm (i)}] Let $f\colon \mathbb{R}^n\to\mathbb{R}^m$ be a map. Clearly, if the problem~\eqref{problem} admits a Pareto solution, say $\bar x$, then the section $[f(\mathbb{R}^n)]_{f(\bar x)}=\{f(\bar x)\}$ is bounded. Thus the condition that $f(\mathbb{R}^n)$ has at least one bounded section is a necessary one for the existence of Pareto solutions of \eqref{problem}.  
													
\item [{\rm (ii)}] By definition, the section $[f(\mathbb{R}^n)]_{\bar{t}}$ is bounded if, and only if, for each sequence $\{x^k\}\subset\mathbb{R}^n$ with $f(x^k)\leq {\bar{t}}$, we have $\{f(x^k)\}$ possesses  a convergent subsequence.
													
\item [{\rm (iii)}] By definition, we have for all $\bar{t} \in (\mathbb{R} \cup \{+ \infty\})^m,$ 
$$\widetilde{K}_{\infty, \le {\bar{t}}}(f)  \subset [\widetilde{K}_{\infty}(f)]_{\bar{t}}, \quad 
{K}_{\infty, \le {\bar{t}}}(f) \subset [{K}_{\infty}(f) ]_{\bar{t}}, \quad 
T_{\infty, \le {\bar{t}}}(f)  \subset [T_{\infty}(f)]_{\bar{t}}.$$
These inclusions may be strict as shown in the following example.
\end{enumerate}
}\end{remark}
												
\begin{example}{\rm
Let $f(x_1, x_2) := (x_1x_2 - 1)^2 + x_1^2$ be a polynomial function in two variables $x_1, x_2.$ We have $f$ is strictly positive on $\mathbb{R}^2$ and so
$$\widetilde{K}_{\infty, \le 0}(f)  = {K}_{\infty, \le 0}(f) = T_{\infty, \le 0}(f)  = \emptyset.$$
														
On the other hand, it is not hard to check that
$$\widetilde{K}_{\infty}(f) = {K}_{\infty}(f)  =  T_\infty(f) = \{0\}.$$
Consequently,
$$[\widetilde{K}_{\infty}(f)]_0 = [{K}_{\infty}(f) ]_0 =  [T_\infty(f)]_0 = \{0\}.$$
}\end{example}
													
\begin{definition}{\rm 
Let $f\colon \mathbb{R}^n\to\mathbb{R}^m$ be a map. We say that:
\begin{enumerate}
	 														
\item [{\rm (i)}] $f$  is  {\em proper at a sublevel} ${\bar{t}} \in\mathbb{R}^m$ if for each compact subset $A\subset [\mathbb{R}^m]_{\bar{t}}$, the inverse image $f^{-1}(A)$ is also compact;
															
\item [{\rm (ii)}] $f$  is {\em proper} if it is proper at every sublevel ${\bar{t}} \in\mathbb{R}^m.$
\end{enumerate}
}\end{definition}
														
\begin{remark}{\rm
By definition, $f$  is proper if, and only if, for each compact subset $A\subset \mathbb{R}^m$, the inverse image $f^{-1}(A)$ is also compact.
}\end{remark}
															
By definition, it is clear that if $f$ is proper at the sublevel ${\bar{t}},$ then 
$$\widetilde{K}_{\infty, \le {\bar{t}}}(f)={K}_{\infty, \le {\bar{t}}}(f)= T_{\infty, \le {\bar{t}}}(f) = \emptyset.$$
The converse does not hold. For example, let $f\colon \mathbb{R}^2\to\mathbb{R}$ be the function defined by $f(x_1, x_2) := x_1 + x_2.$ We see that 
$$\widetilde{K}_{\infty, \le {\bar{t}}}(f)={K}_{\infty, \le {\bar{t}}}(f)= T_{\infty, \le {\bar{t}}}(f)=\emptyset$$
for all ${\bar{t}} \in \mathbb{R}$ but  $f$ is not proper at every sublevel. However, we have  the following result.
															
\begin{theorem}\label{Theorem32}
Let $f\colon\mathbb{R}^n\to \mathbb{R}^m$ be a polynomial map. Assume that there exists $\bar t\in f(\mathbb{R}^n)$ such that the section $[f(\mathbb{R}^n)]_{\bar t}$ is bounded.  Then the following statements are equivalent:
\begin{enumerate}
\item [\rm (i)]  $f$ is proper at the sublevel $\bar t$.
																	
\item[\rm (ii)] $f$ satisfies the Palais--Smale condition at the sublevel $\bar{t}$: $ \widetilde{K}_{\infty, \le \bar{t}}(f)=\emptyset.$
																	
\item[\rm (iii)] $f$ satisfies the weak Palais--Smale condition at the sublevel $\bar{t}$:  ${K}_{\infty, \le \bar{t}}(f)=\emptyset$.
																	
\item[\rm (iv)] $f$ is $M$-tame\footnote{This definition is inspired from the one given in \cite{Nemethi1990}.} at the sublevel $\bar{t}$: $T_{\infty, \le \bar{t}}(f)=\emptyset$.
\end{enumerate}
Furthermore, the set $[f(\mathbb{R}^n)]_{\bar t}$ is closed provided that one of the above equivalent conditions is satisfied.
\end{theorem}
\begin{proof}
The implications (i)$\Rightarrow$(ii)$\Rightarrow$(iii) are immediate from the definitions.

(iii)$\Rightarrow$(iv): This follows from  Proposition~\ref{Proposition32}.

(iv)$\Rightarrow$(i): Arguing by contradiction, assume that $f$ is not proper at the sublevel $\bar t$. Then there exists a compact set $A\subset [\mathbb{R}^m]_{\bar t}$ such that $f^{-1}(A)$ is a non-compact subset in $\mathbb{R}^n$. By the continuity of $f$, the set $f^{-1}(A)$ is unbounded. Thus there exists a sequence $\{x^k\}\subset f^{-1}(A)$ satisfying $\displaystyle\lim_{k\to\infty}\|x^k\|=+\infty$. Since $\{f(x^k)\}\subset A$, we have
$$f(x^k)\leq \bar t\,\,\,\mbox{for all }\,\, k\in\mathbb{N}.$$
For each $k\in\mathbb{N}$, we  consider the problem
\[
\text{\rm Min}_{\,\mathbb{R}^m_+} \{f(x) \,|\, x\in\mathbb{R}^n,  f(x)\leq \bar t\, \, {\rm and } \,\, \|x\|^2=\|x^k\|^2\, \}. 
\]
Since  $\{x\in\mathbb{R}^n \, | \,  f(x)\leq \bar t\, \, {\rm and } \,\, \|x\|^2=\|x^k\|^2\,\}$ is a nonempty compact subset of $\mathbb{R}^n$ and the objective function $f$ is continuous,  the   problem admits a Pareto solution, say  $y^k$. By the Fritz-John optimality conditions \cite[Theorem~7.4]{jahn04}, there are $(\alpha, \beta, \gamma)\in \left(\mathbb{R}^m_+\times\mathbb{R}^m_+\times\mathbb{R}\right) \setminus \{0\}$ such that
$$\sum_{i=1}^m\alpha_i\nabla f_i(y^k)+\sum_{i=1}^m \beta_i\nabla (f_i(\cdot)-\bar t_i)(y^k)+2\gamma y^k=0$$
or, equivalently,
\begin{equation*}
\sum_{i=1}^m(\alpha_i+\beta_i)\nabla f_i(y^k)+2\gamma y^k=0.
\end{equation*}
Put $\lambda_i := \alpha_i + \beta_i$ for $i := 1, \ldots, m$, and $\mu=2\gamma.$ We have
$$\sum_{i=1}^m\lambda_i\nabla f_i(y^k)+\mu y^k=0.$$
Since $(\alpha, \beta, \gamma)\in \left(\mathbb{R}^m_+\times\mathbb{R}^m_+\times\mathbb{R}\right)\setminus \{0\}$, it holds that $(\lambda_1, \ldots, \lambda_m, \mu)\neq 0,$ and so $y^k\in \Gamma (f)$. 
																
We therefore see that the sequence $\{y^k\}$ has the following properties:
\begin{enumerate}
\item [(a)] $\{y^k\}\subset\Gamma (f)$, 
																	
\item[(b)] $\|y^k\|=\|x^k\|\to +\infty$ as $k\to\infty$, and 
																	
\item[(c)] $f(y^k)\leq \bar t$ for all $k\in \mathbb{N}$.	
\end{enumerate}
Now the assumption that $[f(\mathbb{R}^n)]_{\bar t}$ is bounded implies that the sequence $\{f(y^k)\}$ has an accumulation point, say $t\in\mathbb{R}^m$. Clearly, $t\leq \bar t$. Thus $t\in T_{\infty, \le \bar t}(f)$, a contradiction.
																
We now assume that the condition (i) holds. To prove the set $[f(\mathbb{R}^n)]_{\bar t}$ is closed, we need to show that it contains all its limit points. Indeed, let $\{t^k\}\subset [f(\mathbb{R}^n)]_{\bar t}$ be an arbitrary sequence which converges to $t\in\mathbb{R}^m$. Then there exists a sequence $\{x^k\}\subset \mathbb{R}^n$ such that $f(x^k) = t^k \le \bar{t}$ for all $k\in\mathbb{N}$. Since $\displaystyle\lim_{k\to\infty}f(x^k)=t$, there exists a compact set $A\subset\mathbb{R}^m$ such that $\{f(x^k)\}\subset A$. Clearly, the set $A\cap [\mathbb{R}^m]_{\bar t}$ is compact, and so is $f^{-1}\left(A\cap [\mathbb{R}^m]_{\bar t}\right)$ because $f$ is proper at the sublevel $\bar t.$ It follows that the sequence $\{x^k\} \subset f^{-1}\left(A\cap [\mathbb{R}^m]_{\bar t}\right)$ has an accumulation point, say $\bar x\in\mathbb{R}^n$. By the continuity of $f$ and the fact that $\displaystyle\lim_{k\to\infty}f(x^k)=t$, one has $f(\bar x)=t$. Consequently, $t\in f(\mathbb{R}^n)$. Note that $t \le {\bar t}$. Therefore $t\in [f(\mathbb{R}^n)]_{\bar t}$, as required.  $\hfill\Box$
\end{proof} 
															
\section{Existence of Pareto solutions}\label{Section4}
															
The following result concerns the {\em existence} of Pareto solutions for polynomial vector optimization problems.
To the best of our knowledge, the result is new even in the case $m = 1.$ 
															
\begin{theorem}\label{Theorem41}
Let $f\colon\mathbb{R}^n\to \mathbb{R}^m$ be a polynomial map. Assume that there exists $\bar t\in f(\mathbb{R}^n)$ such that the section $[f(\mathbb{R}^n)]_{\bar t}$ is bounded. Then the problem \eqref{problem} admits a Pareto solution, if one of the following equivalent conditions holds:
\begin{enumerate}
\item [\rm (i)]  $f$ is proper at the sublevel $\bar t$.
																	
\item[\rm (ii)] $f$ satisfies the Palais--Smale condition at the sublevel $\bar{t}$: $ \widetilde{K}_{\infty, \le \bar{t}}(f)=\emptyset.$
																	
\item[\rm (iii)] $f$ satisfies the weak Palais--Smale condition at the sublevel $\bar{t}$:  ${K}_{\infty, \le \bar{t}}(f)=\emptyset$.
																	
\item[\rm (iv)] $f$ is $M$-tame at the sublevel $\bar{t}$: $T_{\infty, \le \bar{t}}(f)=\emptyset$.
\end{enumerate}
\end{theorem}
															
\begin{proof}
By Theorem~\ref{Theorem32}, it suffices  to assume that  $f$ is proper at the sublevel $\bar t$. We claim that $[f(\mathbb{R}^n)]_{\bar t}$ is a nonempty compact subset of $\mathbb{R}^m$.  Indeed, let $\{y^k\}$ be an arbitrary sequence in $[f(\mathbb{R}^n)]_{\bar t}$. Then there exists a sequence $\{x^k\}\subset\mathbb{R}^n$ such that $f(x^k) = y^k \leq \bar t$ for all $k\in\mathbb{N}.$
Since the section $[f(\mathbb{R}^n)]_{\bar t}$ is bounded, $\{f(x^k)\}$ has a convergent subsequence. On the other hand, $[f(\mathbb{R}^n)]_{\bar t}$ is a closed set in $\mathbb{R}^n$
due to Theorem~\ref{Theorem32}. Thus $[f(\mathbb{R}^n)]_{\bar t}$ is a nonempty compact set in $\mathbb{R}^n.$ Thanks to \cite[Theorem 1]{borwein83}, the set $f(\mathbb{R}^n)$ has at least one {\em Pareto efficient point}, i.e., there exists $t^* \in f(\mathbb{R}^n)$ such that $f(x) \notin t^* - (\mathbb{R}^m_+\setminus\{0\})$ for all $x\in\mathbb{R}^n$.  This means that the  problem~\eqref{problem} admits a Pareto solution.  The proof is complete.  $\hfill\Box$
\end{proof}
															
As a consequence of Theorem~\ref{Theorem41}, we obtain the following.
\begin{corollary}\label{Corollary41}
Let $f\colon\mathbb{R}^n\to \mathbb{R}^m$ be a polynomial  map such that the section $[f(\mathbb{R}^n)]_t$ is bounded for all $t\in\mathbb{R}^m.$ Then the problem~\eqref{problem} admits a Pareto solution, provided that one of the following equivalent conditions holds:
\begin{enumerate}
\item [\rm (i)]  $f$ is proper.
																	
\item[\rm (ii)] $f$ satisfies the Palais--Smale condition: $\widetilde{K}_{\infty}(f)=\emptyset.$
																	
\item[\rm (iii)] $f$ satisfies the weak Palais--Smale condition: ${K}_{\infty}(f)=\emptyset.$
																	
\item[\rm (iv)] $f$ is $M$-tame: $T_{\infty}(f)=\emptyset.$
\end{enumerate}
\end{corollary}
															
\begin{remark}{\rm
H\`a \cite{ha06} obtained some results on the existence of weak Pareto solutions for multiobjective optimization problems, where the objective function is {\em bounded from below} and  satisfies the so-called {\em $(\mathrm{PS})_1$ condition.} More recently, using the so-called {\em quasiboundedness from below} and {\em refined subdifferential Palais--Smale condition} (RSPS for short), Bao and Mordukhovich \cite{bao07,bao10} studied the existence of relative Pareto solutions for multiobjective optimization problems. Note that the existence theorems established in the papers mentioned do not ensure the existence of Pareto solutions, but only of weak and relative ones.
																	
Regarding to Corollary~\ref{Corollary41} on the existence of {\em Pareto solutions} of the problem~\eqref{problem}, let us mention the following three remarks in comparison with previous results:
																	
\begin{itemize}
\item[$\bullet$] Since the interior of the cone $\mathbb{R}^m_+$ is not empty, all the three relative Pareto solutions introduced in \cite{bao10} agree and in fact they all are weak Pareto solutions. Hence, the results established in \cite{bao07,bao10,ha06} only ensure the existence of weak Pareto solutions.
																		
\item[$\bullet$] Recall that a map $f\colon \mathbb{R}^n\to\mathbb{R}^m$  is said to be {\em bounded from below} if there exists an element $a\in\mathbb{R}^m$ such that
$$f(\mathbb{R}^n)\subset a+\mathbb{R}^m_+.$$ 
Clearly, the map $f$ is bounded from below if, and only if, it is {\em quasibounded from below} (see \cite{bao07,bao10}) in the sense that there exists a bounded set $A\subset \mathbb{R}^m$ such that
$$f(\mathbb{R}^n)\subset A+\mathbb{R}^m_+.$$ 
Furthermore, it follows from definitions that if $f$ is bounded from below, then the section $[f(\mathbb{R}^n)]_t$ is bounded for all $t\in\mathbb{R}^m$. 
The converse is true in the case $m = 1$ but fails to hold in the general case. 
																		
\item[$\bullet$] Let $f \colon \mathbb{R}^n\to\mathbb{R}^m$ a differentiable map. By definition, we can check that the (PS)$_1$ condition\footnote{By a private communication \cite{ha2018}, we would like to note that in the definition of the function $\theta,$ which is used in the (PS)$_1$ condition, the closed unit ball should be replaced by the unit sphere.} (considered in \cite{ha06}) holds for $f$ is equivalent to the fact that $\widetilde{K}_{\infty}(f) = \emptyset,$ which means that
$f$ satisfies the Palais--Smale condition. On the other hand, the (RSPS) condition introduced in \cite{bao10} is stronger than the Palais--Smale condition. To see this, recall that the map $f$ satisfies the (RSPS) condition if every sequence $\{x^k\} \subset \mathbb{R}^n$ such that
$\nu_f(x^k) \to 0$ as $k \to \infty$ contains a convergent subsequence, provided that $\{f(x^k)\}$ is quasibounded from below, i.e.
$$\{f(x^k)\} \subset A + \mathbb{R}^m_+$$ 
for some bounded set $A\subset \mathbb{R}^m.$  By definition, if $f$ satisfies the (RSPS) condition, then it also satisfies the Palais--Smale condition, but the converse fails to hold as can be checked directly for the polynomial 
$$f\colon\mathbb{R}^2 \rightarrow \mathbb{R}^2, \quad (x_1, x_2) \mapsto f(x_1, x_2) := (x_1^2 + x_2^2, x_1^2 - x_2^2).$$
(This polynomial $f$ is proper, and so it satisfies the Palais--Smale condition; furthermore, we have
$$\nu_f(k, 0) = 0  \quad \textrm{ and } \quad f(k, 0) \in \{(0, 0)\} + \mathbb{R}^2_+ \quad  \textrm{ for all } k \in \mathbb{N},$$
which implies that $f$ does not satisfy the (RSPS) condition.)
\end{itemize}
According to the above discussions, it turns out that our results, in the polynomial setting, improve and extend \cite[Theorem 4.1]{ha06}, \cite[Theorem~4]{bao07} and \cite[Theorem~4.4]{bao10}.
}\end{remark}
																
Let us illustrate Theorem~\ref{Theorem41} and  Corollary~\ref{Corollary41} with some examples. 
																
\begin{example}\label{Example41} {\rm
Let us consider the Motzkin polynomial (see \cite{Motzkin1967,HaHV2017})
$$M(x_1, x_2) := x_1^2x_2^4 + x_1^4 x_2^2 - 3 x_1^2x_2^2 + 1.$$
It is not difficult to see that $M(x_1, x_2) \ge 0$ for all $x := (x_1, x_2) \in {\Bbb R}^2.$ Moreover, we have
\begin{itemize}
\item[$\bullet$] If $0 < t < 1,$ then $M^{-1}(t)$ is the union of 4 ovals.
\item[$\bullet$]If $1 < t,$ then $M^{-1}(t)$ is the union of 4 non-compact components.
\item[$\bullet$] The set $M^{-1}(1)$ is non-compact:
$$M^{-1}(1) = \{x_1 = 0 \} \cup \{x_2 = 0 \} \cup \{x_1^2 + x_2^2 = 3 \}.$$
\end{itemize}
Consequently, the polynomial $M$ is proper at the sublevel $\bar t$ if, and only if, $\bar{t} < 1.$ Thanks to Theorem~\ref{Theorem41}, $M$ attains its infimum on $\mathbb{R}^2.$ In fact, we can see that the set of optimal solutions of the problem $\inf_{x \in \mathbb{R}^2} M(x)$ is $M^{-1}(0) = \{(1, 1), (1, -1), (-1, 1), (-1, 1)\}.$ Note that $1 \in T_\infty(M)$ and hence, by Proposition~\ref{Proposition32}, $M$ does not satisfy the Palais--Smale and weak Palais--Smale conditions. Therefore, \cite[Theorem 4.1]{ha06}, \cite[Theorem~4]{bao07}, \cite[Theorem~4.4]{bao10}, and \cite[Theorem~2]{mauhin10}  cannot be applied for this example.
}\end{example}

\begin{example}{\rm  
Let $f\colon\mathbb{R}^3\to \mathbb{R}^2$ be the polynomial map defined by 
$$f(x_1, x_2, x_3)=(x_1^2+x_2^2+x_3^2, x_3^3).$$
It is not hard to see that $f$ is proper and $[f(\mathbb{R}^3)]_t$ is bounded for each $t\in\mathbb{R}^2.$ By Corollary~\ref{Corollary41}, the problem~\eqref{problem} has at least  one Pareto solution. On the other hand, $f$  is not bounded from below, and so, \cite[Theorem 4.1]{ha06}, \cite[Theorem 4]{bao07}, and \cite[Theorem~4.4]{bao10} cannot be applied for this example.
}\end{example} 
																		
The next example shows that if the objective function satisfies one of the equivalent conditions in Theorem~\ref{Theorem41} and $f(\mathbb{R}^n)$ has at least a bounded section, then the set of Pareto solutions of \eqref{problem} is nonempty.
																		
\begin{example}{\rm  
Consider the polynomial map 
$$f \colon\mathbb{R}^3 \to \mathbb{R}^3, \quad (x_1, x_2, x_3) \mapsto (x_1, x_2, M(x_1, x_2) + x_3^2),$$
where $M$ is the Motzkin polynomial defined in Example~\ref{Example41}.
We have 
$$f(\mathbb{R}^3) = \left \{t =(t_1, t_2, t_3) \in \mathbb{R}^3 \, | \, t_3 \ge M(t_1, t_2) \right\}$$
and the section $[f(\mathbb{R}^3)]_t$ is unbounded for every $t = (t_1, t_2, t_3) \in \mathbb{R}^3$ with $t_3 \ge 1.$ 
On the other hand, if we take $\bar{t} := (1, 1, 0) \in \mathbb{R}^3$ then the section $[f(\mathbb{R}^3)]_{\bar t}$ is bounded and $f$ is proper at the sublevel $\bar t.$
Thus, by Theorem~\ref{Theorem41}, the problem \eqref{problem} has at least  one Pareto solution. However, $f$  is not bounded from below, and so \cite[Theorem 4.1]{ha06}, \cite[Theorem 4]{bao07}, and \cite[Theorem~4.4]{bao10} cannot be applied for this example.
}\end{example} 
																			
In the rest of the paper we shall give some classes of vector optimization problems, which satisfy the conditions in Theorem \ref{Theorem41}. We start with the class of {\em linear vector optimization problems.}
																			
\begin{corollary}[{compare \cite[Theorem~6.5]{Ehrgott06}}]
Let $f_i(x) := \langle a_i, x\rangle +b_i$, where $a_i\in\mathbb{R}^n$ and   $b_i \in\mathbb{R}$ for all $i = 1, \ldots, m.$ 
Assume that the set of vectors $\{a_1, \ldots, a_m\}$ is linearly independent.  If $f(\mathbb{R}^n)$ has a bounded section, then \eqref{problem} admits a  Pareto solution. 
\end{corollary}
\begin{proof}
For each $x\in\mathbb{R}^n$, we have $$\nu_f(x)=\min_{\sum_{i=1}^m|\lambda_i|=1}\left\|\sum_{i=1}^m \lambda_i\nabla f_i(x)\right\|=\min_{\sum_{i=1}^m|\lambda_i|=1}\left\|\sum_{i=1}^m \lambda_i a_i\right\|.$$
By the compactness of the set $\{\lambda=(\lambda_1, \ldots, \lambda_m)\in\mathbb{R}^m\,\,|\,\, \sum_{i=1}^m|\lambda_i|=1\}$, there exists $\bar\lambda=(\bar\lambda_1, \ldots, \bar\lambda_m)\in\mathbb{R}^m$ with $\sum_{i=1}^m|\bar\lambda_i|=1$ such that
$$\nu_f(x)=\left\|\sum_{i=1}^m \bar\lambda_i a_i\right\|=:\delta.$$
Since the set of vectors $\{a_1, \ldots, a_m\}$ is linearly independent and $\sum_{i=1}^m|\bar\lambda_i|=1$, we have $\nu_f(x)=\delta>0$ for all $x\in\mathbb{R}^n$. Consequently, 
$$\widetilde{K}_{\infty}(f)={K}_{\infty}(f)={T}_{\infty}(f)=\emptyset.$$
Thanks to Theorem~\ref{Theorem41}, \eqref{problem} admits a  Pareto solution.   $\hfill\Box$
\end{proof}
																			
We finish this section with a generic class of polynomial vector optimization problems having at least one Pareto solution. To do this, we begin with some definitions.
If $\kappa = (\kappa_1, \ldots, \kappa_n) \in \mathbb{N}^n,$ we denote by $x^\kappa$ the monomial $x_1^{\kappa_1} \cdots x_n^{\kappa_n}$ and by $| \kappa|$ the sum $\kappa_1 + \cdots + \kappa_n.$ Note that when $\kappa = (0, \ldots, 0),$ $x^\kappa = 1.$
																			
Let $f \colon {\Bbb R}^n \to {\Bbb R}$ be a polynomial function. Suppose that $f$ is written as $f = \sum_{\kappa} a_\kappa x^\kappa.$ 
By the {\em Newton polyhedron at infinity} of $f,$ denoted by $\mathcal{N}(f),$ we mean the convex hull in $\mathbb{R}^n$ of the set $\{\kappa \ | \ a_\kappa  \ne 0\} \cup \{0\}.$
The polynomial $f$ is said to be {\em convenient} if $\mathcal{N}(f)$ intersects each coordinate axis in a point different from the origin. The {\em Newton boundary at infinity} of $f$, denoted by $\mathcal{N}_{\infty}(f),$ is defined as the set of the faces of $\mathcal{N}(f)$ which do not contain the origin $0$ in ${\Bbb R}^n.$ For each face $\Delta$ of $\mathcal{N}_\infty(f),$  we define the {\em principal part of $f$ at infinity with respect to $\Delta,$} denoted by $f_\Delta,$ as the sum of the terms $a_\kappa x^\kappa$ such that $\kappa \in \Delta.$
																			
Let $f := (f_1, \ldots, f_m) \colon {\Bbb R}^n \rightarrow {\Bbb R}^m$ be a polynomial map. We say that $f$ is {\em convenient} if all its components $f_i$ are convenient. 
Let $\mathcal{N}(f)$ denote the Minkowski sum $\mathcal{N}(f_1) + \cdots + \mathcal{N}(f_m),$ i.e., 
$$\mathcal{N}(f)  := \{\kappa^1 + \cdots + \kappa^m \ | \ \kappa^i \in \mathcal{N}(f_i)  \  \textrm{ for all } \ i = 1, \ldots, m\}.$$
We denote by $\mathcal{N}_\infty(f)$ the set of faces of $\mathcal{N}(f)$ which do not contain the origin $0$ in ${\Bbb R}^n.$ Let $\Delta$ be a face of the $\mathcal{N}(f).$ According to \cite[Lemma~2.1]{Dinh2012-2}, we have a unique decomposition $\Delta = \Delta_1 + \cdots + \Delta_m,$ where $\Delta_i$ is a face of $\mathcal{N}(f_{i})$ for $i = 1, \ldots, m.$ We denote by $f_\Delta$ the map $(f_{1, \Delta_1}, \ldots, f_{m, \Delta_m})  \colon {\Bbb R}^n \rightarrow {\Bbb R}^m$, where $f_{i, \Delta_i}$  is the principal part of $f_i$ at infinity with respect to $\Delta_i.$
																			
\begin{definition}[see \cite{Khovanskii1978,Kouchnirenko1976}] {\rm
We say that $f = (f_1, \ldots, f_m)$ is {\em Khovanskii non-degenerate at infinity} if, and only if, for any face $\Delta$ of $\mathcal{N}_\infty(f)$ and for all $x \in ({\Bbb R} \setminus \{0\})^n \cap f_\Delta^{-1}(0),$ we have
\begin{equation*}
\textrm{rank}
\begin{pmatrix}
x_1\frac{\partial f_{1, \Delta_1}}{\partial x_1}(x) & \cdots & x_n \frac{\partial f_{1, \Delta_1}}{\partial x_n}(x) \\
\vdots & \cdots & \vdots \\
x_1\frac{\partial f_{m, \Delta_m}}{\partial x_1}(x) & \cdots & x_n \frac{\partial f_{m, \Delta_m}}{\partial x_n}(x)
\end{pmatrix} = m.
\end{equation*}
}\end{definition}
																				
\begin{remark}{\rm
We should emphasize that the class of polynomial maps (with fixed Newton polyhedra), which are non-degenerate at infinity, is an open and dense semi-algebraic set  in the corresponding Euclidean space of data. For more details, see \cite{Dinh2012-2} and \cite[Theorem~5.2]{HaHV2017}.
}\end{remark}
																					
We now present an efficient consequence of Theorem~\ref{Theorem41} ensuring the existence of Pareto solutions for the class of polynomials which are convenient and Khovanskii non-degenerate at infinity.
																					
\begin{corollary}[Frank--Wolfe type theorem] 
Let $f \colon\mathbb{R}^n\to \mathbb{R}^m$ be a polynomial map. Suppose that $f$ is convenient and Khovanskii non-degenerate at infinity. If $f(\mathbb{R}^n)$ has a bounded section, then~\eqref{problem} admits a  Pareto solution. 
\end{corollary}
\begin{proof} 
Thanks to \cite[Theorem~3.2]{Dinh2014-2}, $\widetilde{K}_{\infty}(f)=\emptyset$. Then the assertion follows  immediately from Theorem~\ref{Theorem41}.  $\hfill\Box$
\end{proof}
																					
\section{Conclusions}\label{Section5}
																					
In this paper, we obtained some results on the existence of Pareto solutions of  polynomial vector optimization problems. Some relationships between Palais--Smale conditions, $M$-tameness, and properness are also examined. Further research for optimization problems with constraints will be studied in future work.

\begin{acknowledgements}
The authors would like to thank the three referees for careful reading and constructive comments. A part of this work was done while the second and third authors were visiting Department of Applied Mathematics, Pukyong National University, Busan, Korea in September 2016. These authors would like to thank the department for  hospitality and support during their stay.
\end{acknowledgements}



\end{document}